\documentclass[12pt,a4paper]{amsart}
\usepackage{amssymb,latexsym,amsmath,amscd,graphicx,url,color,comment}
\numberwithin{equation}{section}
\let\oldlabel=\label
\def\prellabel{\marginparsep=1em\marginparwidth=44pt
    \def\label##1{\oldlabel{##1}\ifmmode\else\ifinner\else
         \marginpar{{\footnotesize\ \\ \tt
                    ##1}}\fi\fi}}

\theoremstyle{plain}
\newtheorem{theorem}{\bf Theorem}[section]

\newtheorem{Lemma}[theorem]{Lemma}
\newtheorem{Proposition}[theorem]{Proposition}

\newtheorem{Remark}[theorem]{Remark}

\newtheorem{Definition}[theorem]{Definition}
\newtheorem{Theorem}[theorem]{Theorem}

 \newcommand{\A}{{\mathcal A}}

\newcommand{\K}{{\mathcal K}}

\newcommand{\M}{{\mathcal M}}

\newcommand{\NN}{{\mathbb N}}

\newcommand{\ZZ}{{\mathbb Z}}

\newcommand{\KK}{{\mathbb K}}

\DeclareMathOperator{\Brad}{B-rad}

\DeclareMathOperator{\CS}{CS}

\DeclareMathOperator{\HS}{HS}

\DeclareMathOperator{\GL}{GL}

\DeclareMathOperator{\gin}{gin}
\DeclareMathOperator{\inid}{in}

\DeclareMathOperator{\pol}{pol}

\title{Radical support for multigraded ideals}

\author{A. Conca}
\address{Dipartimento di Matematica, 
Universit\`a di Genova, Via Dodecaneso 35, 
I-16146 Genova, Italy}
\email{conca@dima.unige.it}

\author{E. De Negri}
\address{Dipartimento di Matematica, 
Universit\`a di Genova, Via Dodecaneso 35, 
I-16146 Genova, Italy}
\email{denegri@dima.unige.it}

\author{E. Gorla}
\address{Institut de Math\'ematiques, Universit\'e de Neuch\^atel, Rue Emile-Argand 11, 
CH-2000 Neuch\^atel, Switzerland}  
\email{elisa.gorla@unine.ch}

\thanks{The first and the second authors were partially supported by  PRIN 2020355B8Y  ``Squarefree Gr\"obner degenerations, special varieties and related topics"  and by GNSAGA-INdAM}


\begin{document}

\begin{abstract}
Can one tell if an ideal is radical just by looking at the degrees of the generators? In general, this is hopeless. However, there are special collections of degrees in multigraded polynomial rings, with the property that any multigraded ideal generated by elements of those degrees is radical. We call such a collection of degrees a radical support. In this paper, we give a combinatorial characterization of radical supports. Our characterization is in terms of properties of cycles in an associated labelled graph. We also show that the notion of radical support is closely related to that of Cartwright-Sturmfels ideals. In fact, any ideal generated by multigraded generators whose multidegrees form a radical support is a Cartwright-Sturmfels ideal. Conversely, a collection of degrees such that any multigraded ideal generated by elements of those degrees is Cartwright-Sturmfels is a radical support.
\end{abstract}

\maketitle

\section{Radical supports}
\label{RadSup}
Let  $n$ be a positive integer    and $m=(m_1,\dots, m_n)\in  \NN_{>0}^n$.  
Let 
$$S(m)=K[x_{ij}\mid 1\leq j\leq n,\ 1\leq i\leq m_j]$$
 be a polynomial ring over a field $K$ endowed with the standard $\ZZ^n$-grading induced by setting $\deg(x_{ij})=e_j$, where $e_j\in\ZZ^n$ is the $j$-th standard basis vector. We will denote $S(m)$ only by $S$ if there is no danger of confusion.

Given $A\subseteq [n]=\{1,2,\dots, n\}$ we will denote by $S_A$ the $\ZZ^n$-graded homogeneous component of degree $\sum_{i\in A} e_i$ of $S$. An element $f\in S_A$ will be said to have  multidegree $A$. 

\begin{Remark} 
\label{rema1} 
Let $\emptyset \neq A\subseteq [n]$ and let $f\in S_A$.    If $f\neq 0$  then it  cannot have multiple factors in its factorization into irreducible factors.  Hence  $(f)$ is radical and the same obviously holds if $f=0$.  
\end{Remark} 

More generally we have: 

\begin{Lemma} 
\label{facile}
Consider a collection $\A=\{ A_1,\dots, A_s\}$ of non-empty subsets of $[n]$ such that $A_i\cap A_j=\emptyset$ for all $i\neq j$. Then for every $m\in \NN_{>0}^n$, for every field $K$  and for every choice of $f_1\in S_{A_1}, \dots, f_s\in S_{A_s}$ the ideal $I=(f_1,\dots, f_s)$ of $S$ is radical. 
\end{Lemma} 
\begin{proof} 
We introduce a term order $\sigma$ in $S$. Then the leading term of the nonzero $f_i$'s are pairwise coprime and squarefree.  Hence $f_1,\dots, f_s$ form a Gr\"obner basis with $\inid(I)$ radical and  $I$ is radical as well.  \end{proof} 

For generalities on  Gr\"obner bases and the transfer of properties under Gr\"obner deformation we refer the reader to \cite{BC3} or, for a complete picture, to  the forthcoming \cite{BCRV}. 

The simple assertion of Lemma \ref{facile}  suggests the following definition. 

\begin{Definition}
\label{def1} 
 A collection (repetitions are allowed) $\A=\{ A_1,\dots, A_s\}$  of non-empty subsets of $[n]$ is a radical support  with respect to a field $K$ if for  every $m=(m_1,\dots, m_n)\in  \NN_{>0}^n$ and for every choice of $f_1\in S_{A_1}, \dots, f_s\in S_{A_s}$ the ideal $I=(f_1,\dots, f_s)$ of $S$ is radical. Furthermore we say that 
 $\A$ is a radical support if it is a radical support for every field $K$.  
\end{Definition}

Our goal is to provide a combinatorial characterization of radical supports. We stress that in Definition \ref{def1} the required property should hold for every $m\in \NN_{>0}^n$. For example if $m=(1,1,\dots,1)$  then  $\ZZ^n$-graded ideals of $S$ are indeed monomial ideals and hence $I$ is radical if - and only if - its minimal generators have squarefree degrees (independently of the field).  

As we will see, the characterization of radical supports is related to the notion of Cartwright-Sturmfels ideals. This notion has its roots in the work of Bernstein, Boocher, Brion, Cartwright, Conca, Sturmfels, Villarreal and Zelevisky \cite{BZ, Bo, Br,C,CS, SZ,V} among others. The theory of Cartwright-Sturmfels ideals has been developed in a series of papers by the authors of the present manuscript \cite{CDG1,  CDG2, CDG3, CDG4,  CDG5} and studied further by Conca and Welker in \cite{CW}. Here we quickly revise the definition and the main features of  Cartwright-Sturmfels ideals referring the readers to \cite{CDG1,  CDG2, CDG3, CDG4,  CDG5} for more general statements and for  the proofs.

\section{Cartwright-Sturmfels ideals} 

We follow the notation of the previous section.     
The group $G=\GL_{m_1}(K)\times\cdots\times\GL_{m_n}(K)$ acts  on $S$ as the group of multigraded $K$-algebra automorphisms.  Let $B=B_{m_1}(K)\times \cdots  \times B_{m_n}(K)$ be the Borel   subgroup of $G$, consisting of the upper triangular  matrices in $G$. 
Let $I$ be   a multigraded ideal of $S$ and let $\sigma$ be a term order  on $S$ with $x_{ij}>x_{kj}$  for all $j$ and all $i<k$.   

If  $K$ is infinite we may define,  as in  the standard $\ZZ$-graded situation, the (multigraded) generic initial ideal $\gin(I)$ of $I$ with respect to $\sigma$ as  $\inid_\sigma(g(I))$ with $g$ a ``generic" element in $G$. It turns out that $\gin(I)$ is  Borel fixed, i.e., fixed by the action of $B$. 

Radical Borel-fixed ideals play an important role in the definition of Cartwright-Sturmfels ideals.

\begin{Definition}  
We let $\Brad(S)$  be the set of radical monomial ideals $J$ of $S$ such that for every monomial $f\in S$ and $i,j$ such that  $x_{ij}f$ is a generator of $J$ one has $x_{kj}f\in J$ for all $k<i$. 
\end{Definition} 

The ideals in  $\Brad(S)$ are Borel-fixed. Indeed, at least if $K$ is infinite, $\Brad(S)$ is the set of Borel-fixed radical ideals, hence the name. 
 
 Let $M$ be a finitely generated $\ZZ^n$-graded $S$-module and assume for simplicity that $M_a=0$ if $a\not\in \NN^n$.  
 
The multigraded Hilbert series $H_M(z_1,\dots,z_n)$ of $M$  has a rational expression 
 $$\HS_M(z_1,\dots,z_n)=\frac{\KK_M(z_1,\dots,z_n)}{\prod_{i=1}^n(1-z_i)^{m_i}}.$$
Here  $\KK_M(z_1,\dots,z_n)$ is a polynomial with  coefficients in $\ZZ$ known as the $K$-polynomial of $M$. The dual $K$-polynomial $\KK^*_M(z_1,\dots,z_n)$ of $M$ is defined as
 $$\KK^*_M(z_1,\dots,z_n)=\KK_M(1-z_1,\dots,1-z_n).$$
If $n$ is clear from the context we will use  $\HS_M(z)$ for $\HS_M(z_1,\dots,z_n)$ and similarly will use $\KK_M(z)$ and  $\KK^*_M(z)$.  
\begin{Definition}\label{sameHS}
Let $I$ be a multigraded ideal of $S$. Then $I$ is Cartwright-Sturmfels if there exists $J\in \Brad(S)$ such that $\HS_{S/I}(z)=\HS_{S/J}(z)$. 
\end{Definition}  
We denote by $\CS(S)$ the family of Cartwright-Sturmfels ideals of $S$.
It turns out that the ideal $J$ that appears in Definition \ref{sameHS} is uniquely determined by (the multigraded Hilbert series of) $I$, see \cite[Theorem 3.5]{CDG1}. This leads to the following characterization. 
\begin{Proposition}\label{ginRad}
Assume that $K$ is infinite. Then $I$ is Cartwright-Sturmfels if and only if $\gin(I) \in \Brad(S)$.
\end{Proposition}  
For every $j\in [n]$ we rename $y_j$ the variable $x_{1j}$ and let $T=K[y_1,\dots, y_n] \subseteq S$ with induced (fine) $\ZZ^n$-graded structure, i.e. $\deg y_{j}=e_j\in \ZZ^n$. Let $\M(T,m)$ be the set of the monomial ideals of $T$ generated by monomials in the $y$'s, whose exponent vector is bounded from above by $m=(m_1,\dots,m_n)$. 

The two sets  $\Brad(S)$ and $\M(T,m)$  of monomial ideals are in one-to-one correspondence via the map

\begin{itemize}
 \item [(1)] $\psi: \M(T,m) \to  \Brad(S)$ that sends $E \in \M(T,m)$ to $\psi(E)=J=\pol(E)^*$. Here $\pol(E)$ is the polarization of $E$  and the star $^*$  denotes the Alexander dual. 
\end{itemize} 

When $K$ is infinite, the inverse of $\psi$ can be defined as  follows.

\begin{itemize}
\item [(2)] $\phi:\Brad(S) \to \M(T,m)$ sends $J\in \Brad(S)$ to $\phi(J)=E$, where $E \in \M(T,m)$ is uniquely determined by the property $\gin(J^*)=ES$.  Here $J^*$ is the Alexander dual of $J$ and  $ES$ is the extension of $E$ to $S$.  
\end{itemize} 
   
When $K$ is finite, one can take an infinite extension $F$ of $K$. The map $\phi$ is defined on $F$ and $\phi(J)=E$ is an ideal generated by monic monomials. The same monomials generate the ideal $\psi^{-1}(J)\in \M(T,m)$. 

\begin{Proposition}\label{JandE} 
Let $K$ be an arbitrary field. Given $J\in \Brad(S)$ and $E\in \M(T,m)$ one has that   $J$ corresponds to $E$ in the bijective correspondence  above  if and only if $\KK^*_{S/J}(z)=\KK_E(z)$.  
\end{Proposition} 
 
Furthermore we have

\begin{Lemma} 
\label{CS-K}
Let $m=(m_1,\dots,m_n), q=(q_1,\dots,q_n)\in \NN_{>0}^n$. Let $K$ and $L$ be fields, let
$I\subseteq S=K[x_{ij} : j\in [n] \mbox{ and } i\leq m_j]$ and $J\subseteq R=L[x_{ij} : j\in [n] \mbox{ and } i\leq q_j]$ be  multigraded ideals  such that $\KK_{S/I}(z)=\KK_{R/J}(z)$. Then $I\in \CS(S)$ if and only if $J\in \CS(R)$. 
\end{Lemma}

\begin{proof} 
Assume $I\in \CS(S)$. Then there exists $H\in \Brad(S)$ such that $I$ and $H$ have the same Hilbert series, in particular $\KK_{S/I}(z)=\KK_{S/H}(z)$. Then set  $r=(r_1,\dots, r_n)$ with $r_j= \max\{m_j,q_j\}$ for all  $j\in [n] $ and $T=L[x_{ij} : j\in [n] \mbox{ and } i\leq r_j]$. Then we have the inclusion $R\subseteq T$. 
We may consider  the ideal $H'$ of $T$ generated by the monomial generators of $H$. 
Since  $H\in \Brad(S)$ we have $H'\in \Brad(T)$.  We can consider the extensions  $JT$ of $J$. Obviously the $K$-polynomial  does not change under a polynomial extension and under the passage from $H$ to $H'$. Hence 
$$\KK_{T/JT}(z)=\KK_{R/J}(z)=\KK_{S/I}(z)=\KK_{S/H}(z)=\KK_{T/H'}(z).$$
Hence $JT$ has the same Hilbert series of $H'$, that is, $JT \in \CS(T)$. 
It remains to prove that $JT \in \CS(T)$ implies $J\in \CS(R)$. First assume that $L$ is infinite. 
The computation of generic initial ideals commute  with polynomial extensions, i.e. $\gin(JT)=\gin(J)T$. Since  $\gin(JT)\in \Brad(T)$ we have that $\gin(J)\in \Brad(R)$ and hence $J\in \CS(R)$.  If $L$ is finite then we can consider an infinite extension  of $L$, compute the gin of $J$ in with coefficients in the extension and consider the outcome in the original polynomial ring $R$. In this way, repeating the argument above, we obtain an ideal in $\Brad(R)$ with the Hilbert series of $J$. Hence $J\in \CS(R)$ in this case as well. 
 \end{proof} 

Cartwright-Sturmfels ideals remains Cartwright-Sturmfels under  arbitrary $\ZZ^n$-graded linear section. This is proved in \cite{CDG2} under the assumption that the base field is infinite but the result holds in general as the reader can easily check. 
 
\begin{Proposition}\label{section}
 Let $L$ be an ideal of $S$ that is generated by  $\ZZ^n$-graded  linear forms. We identify $R=S/(L)$  with a polynomial ring with the induced $\ZZ^n$-graded structure and $J=I+(L)/(L)$ with an ideal of $R$. If $I\in \CS(S)$, then   $J\in \CS(R)$.
\end{Proposition}

We conclude with a definition: 
  \begin{Definition}
\label{def2} 
 A collection (repetitions are allowed) $\A=\{ A_1,\dots, A_s\}$  of non-empty subsets of $[n]$ is a  Cartwright-Sturmfels support   if for  every $m=(m_1,\dots, m_n)\in  \NN_{>0}^n$ and for every choice of $f_1\in S_{A_1}, \dots, f_s\in S_{A_s}$ the ideal $I=(f_1,\dots, f_s)$ of $S$ is Cartwright-Sturmfels. 
\end{Definition} 
  
\section{Support for regular sequences and the graph associated to the support}  

For later applications, we establish some simple facts. 

\begin{Lemma}\label{suppreg} 
Let   $\A=\{ A_1,\dots, A_s\}$ be a collection of non-empty subsets  of $[n]$. 
The following are equivalent:
 \begin{itemize} 
 \item[(1)]  There exist  $f_1\in S_{A_1}, \dots, f_s\in S_{A_s}$ such that  $f_1,\dots, f_s$ is a regular sequence and $(f_1,\dots, f_s)$ is radical. 
 \item[(2)]  There exist  $f_1\in S_{A_1}, \dots, f_s\in S_{A_s}$ such that  $f_1,\dots, f_s$ is a regular sequence.
 \item[(3)] For every $j\in [n]$ we have $|\{ v : j\in A_v\}|\leq m_j$. 
 \end{itemize} 
\end{Lemma}

\begin{proof} 
That condition (1) implies condition (2) is obvious. For (2) implies (3) suppose there exist  $f_1\in S_{A_1}, \dots, f_s\in S_{A_s}$ such  $f_1,\dots, f_s$ is a regular sequence.  For a given $j\in [n]$ the ideal  $( f_v : j\in A_v)$ is generated by  regular sequence and contained in the ideal generated by $S_{e_j}$.  Hence we have that $|\{ v : j\in A_v\}|\leq m_j$.  Finally for (3) implies (1)  for every $v\in [s]$ and every $j\in A_v$ we set 
$$k_{j,v}=|\{ w: w\leq v \mbox{ and } j\in A_w\}|.$$ 
By assumption $|\{ v : j\in A_v\}|\leq m_j$ and hence $k_{j,v}\leq m_j$.  Therefore  we may set 
$$f_v=\prod_{j \in A_v} x_{k_{j,v},j} \in S_{A_v}$$
for every $v\in [s]$. Notice that $k_{j,v}\geq 1$, since $j\in A_v$. Moreover, $x_{k_{j,v},j}=x_{k_{i,w},i}$ implies $i=j$ and if $v>w$, then $k_{j,v}>k_{j,w}$. Hence  $f_1,\dots, f_s$ is a set of pairwise coprime and squarefree monomials.   
 In other words,  $f_1,\dots, f_s$ is a (monomial) regular sequence that generates a radical ideal. 
\end{proof} 

\begin{Lemma}\label{nonrad1} 
Suppose $I=(f_1,\dots, f_c)$ is a non-radical ideal of a ring $R$. Then for every $1\leq t\leq c$ the ideal 
$J=(f_1,\dots, f_{t-1}, xf_{t}, xf_{t+1}, \dots, xf_c)$ of the polynomial extension $R[x]$ is not radical. 
\end{Lemma} 

\begin{proof} Consider on $R[x]$ the graded structure associated to the assignment $\deg x=1$ and $\deg a=0$ for all $a\in R$. Then $J$ is homogeneous and its homogeneous components $J_i$ with $i>0$ are all equal to $I$. Let $a\in R$ be such that $a\not\in I$ and $a^2\in I$. Then $ax\not\in J$ and $(ax)^2 \in J$ so that  $J$ is not radical. 
\end{proof} 

\begin{Lemma}\label{subcolle} 
If $\A=\{ A_1,\dots, A_s\}$  is a radical support with respect to a field $K$ then every subcollection of $\A$  is a radical support  with respect to $K$. 
\end{Lemma} 
 
\begin{proof}  
It suffices to observe that in Definition \ref{def1} some of the polynomials  $f_i\in S_{A_i}$ can be taken equal to $0$. 
\end{proof} 

Given a collection $\A=\{ A_1,\dots, A_s\}$ of subsets  of $[n]$ we define a non-oriented graph $G(\A)$ with possibly multiple   edges and no loops as follows. The vertices of $G(\A)$ are labelled by elements in $[s]=\{1,2,\dots,s\}$  and we put an edge labelled with $j\in [n]$ between vertex $v$ and vertex $w$    if $j\in A_v\cap A_w$. This will be denoted by  
$$\includegraphics[width=0.7\textwidth, trim={135 650 65 135}]{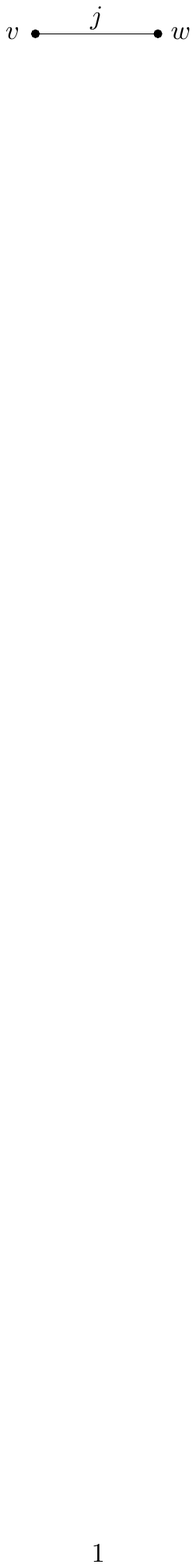} $$
A  cycle in $G(\A)$  is a sequence of vertices and edges  of $G(\A)$ 
$$\includegraphics[width=0.7\textwidth, trim={135 550 65 130}]{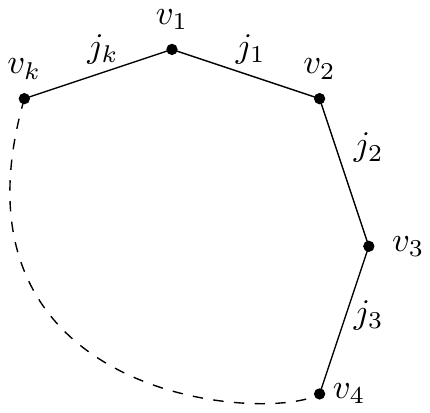}$$ with $v_1,\dots, v_k$ distinct. 
A $k$-cycle is  a cycle  which involves $k$ edges and $k$ distinct vertices. 

We are interested in cycles with non-constant edge labels, i.e.,  cycles with edge labels $j_1,\dots, j_k$ such that  $j_a\neq j_b$ for at least a pair $a,b$ in $[k]$. We observe the following. 
 
\begin{Lemma}\label{edgelabel}
If $G(\A)$  has  a cycle with non-constant edge labels, then it has a cycle where all the edge labels are distinct. 
\end{Lemma} 

\begin{proof} 
By assumption there is a cycle, say a $k$-cycle,  with non-constant edge labels, call it $C$. If $k=2$ then $C$ has distinct edge labels and we are done. So we may assume $k>2$.  We may also assume that $C$ has some repeated  edge labels (otherwise, we are done).  Say the edge label $1$ appears more than once in $C$. Since the edge labels are not constant, the edge label $1$ must  be adjacent  at least once to a different edge label, say $2$. Up to  a ``rotation" we may hence  assume that $C$  looks like this: 
$$\includegraphics[width=0.7\textwidth, trim={135 550 65 130}]{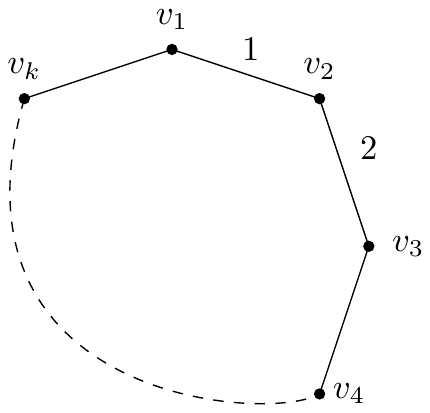}$$
We know that there is at least another label equal to $1$, say from $v_p$ to $v_{p+1}$, so that  
$$\includegraphics[width=0.7\textwidth, trim={135 550 65 130}]{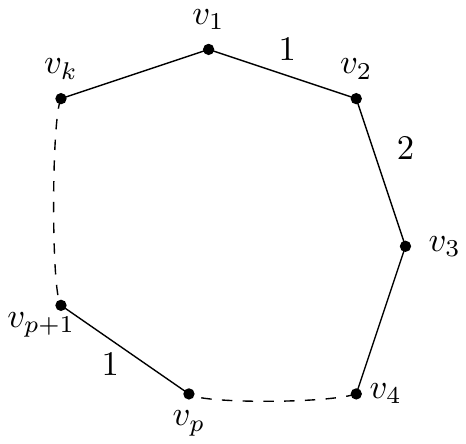}$$
with $3\leq p\leq k$ and $p+1=1$ if $p=k$. Then $1\in A_{v_2} \cap A_{v_p}$, hence the edge 
$$\includegraphics[width=0.7\textwidth, trim={135 650 65 135}]{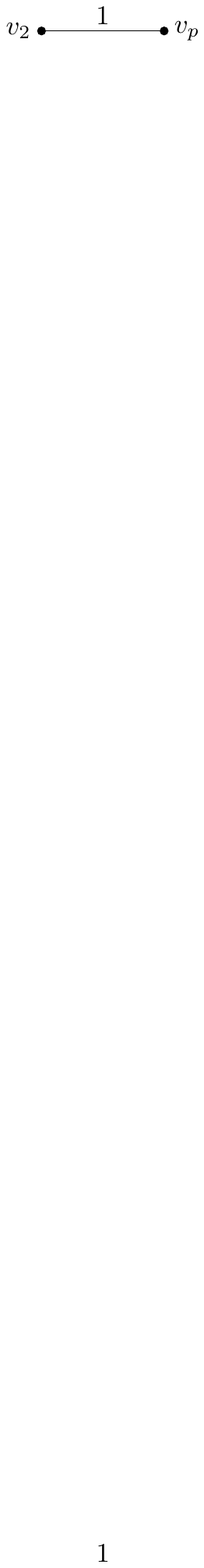} $$
is in $G(\A)$.  Therefore we have a  shorter cycle  in $G(\A)$
$$\includegraphics[width=0.7\textwidth, trim={135 550 65 130}]{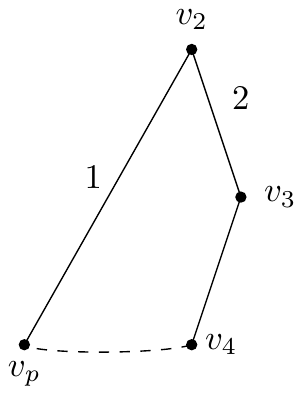}$$
with non-constant edge labels. Iterating the process, we produce a cycle where all the edge labels are distinct. 
\end{proof} 
  
\begin{Lemma}\label{nonrad} 
Let $p\in \NN$ with $p\geq 2$ and $K$ be any field. Set $S=K[x_1,\dots,x_p, y_1,\dots, y_p]$ with $\ZZ^p$-graded structure given by $\deg x_i=\deg y_i=e_i$ for every $i=1,\dots, p$. Let 
$$I=(x_{i+1}y_i-x_iy_{i+1}  : i=1,2,\dots, p-1)+(y_1y_p).$$
Then $I$ is not radical. In particular,    
$$\A=\{ \{1,2\}, \{2,3\}, \{3,4\},\dots, \{p-1,p\}, \{1,p\}  \}$$
is not a radical support with respect to $K$. 
\end{Lemma} 

\begin{proof} 
A straightforward application of Buchberger's Algorithm shows  that the reduced Gr\"obner basis of $I$ with respect to the degree reverse lexicographic term order induced by the total order on the variables 
$$x_1>\dots>x_p >y_1>\dots >y_p$$
is obtained by adding to the given generators the monomials
$$(x_1y_p) y_2, (x_1y_p) x_2y_3,  (x_1y_p) x_2x_3y_4, \dots,  (x_1y_p) x_2x_3\cdots x_{p-1}y_p.$$
Therefore $(x_1y_p) x_2x_3\cdots x_{p-1}y_p\in I$ and $x_1 x_2x_3\cdots x_{p-1}y_p \not\in I$. This shows that $I$ is not radical, hence $\A$ is not a radical support with respect to $K$.    
\end{proof}

\section{The main result} 
 
We are ready to formulate and prove our main result. 

\begin{Theorem} 
\label{main}
Given a collection $\A=\{ A_1,\dots, A_s\}$ of non-empty subsets  of $[n]$ we have that the following conditions are equivalent. 
\begin{itemize} 
\item[(1)] $\A$ is a radical support. 
\item[(2)] $\A$ is a radical support for at least one field $K$. 
\item[(3)] The graph $G(\A)$  has no  cycles with distinct edge labels. 
\item[(4)] The graph $G(\A)$ has no cycles with non-constant edge labels. 
\item[(5)]  There exist a field $K$, $m=(m_1,\dots,m_n)$ and a  regular sequence $f_1,\dots, f_s$ of degrees $\A$ in $S$ such that the ideal $(f_1,\dots, f_s)$ is Cartwright-Sturmfels.
\item[(6)]  For every  field $K$, every $m=(m_1,\dots,m_n)$ such that $m_j\geq |\{v\in[s] : j\in A_v\}|$ for all $j$, and every $f_1,\dots, f_s$ regular sequence of degrees $\A$ in $S$ the ideal $(f_1,\dots, f_s)$ is Cartwright-Sturmfels.
 \item[(7)]  $\A$ is a Cartwright-Sturmfels support. 
 \end{itemize} 
\end{Theorem} 

\begin{proof} 
It is clear that (1) implies (2). We prove that (2) implies (3) by contradiction. That is, we assume that (2) holds and that $G(\A)$ has a cycle with distinct edge labels and derive a contradiction. 
Renaming vertices and edges if needed, we may hence assume that $G(\A)$ contains the  $p$-cycle 
$$\includegraphics[width=0.7 \textwidth, trim={135 550 65 130}]{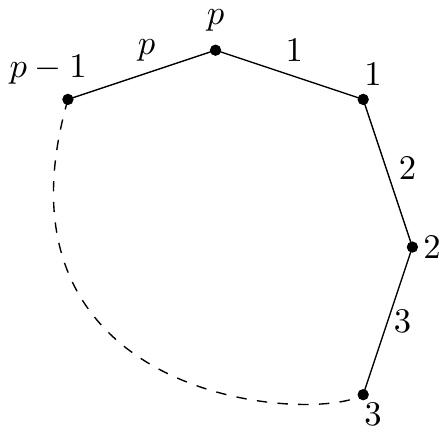}$$
with $p>1$. 
By \ref{subcolle} the subcollection $\{A_1,\dots, A_p\}$ is a radical support for $K$. Notice that by construction  
\begin{equation}\label{inclu}
A_1\supseteq \{1,2\}, \  A_2\supseteq \{2,3\}, \  \dots,  A_{p-1} \supseteq \{p-1,p\} \mbox{  and }  A_p\supseteq \{1,p\}. 
\end{equation} 
  If in (\ref{inclu}) we have all equalities then by  \ref{nonrad} we have polynomials $f_1,f_2,\dots, f_p$ of degree $\{1,2\}, \{2,3\}, \dots, \{p-1,p\}, \{1,p\}$ in a multigraded polynomial ring over $K$ that generate a non-radical ideal  and hence we have a contradiction.  
 If instead some inclusions in (\ref{inclu}) are strict   we still consider the polynomials $f_1,\dots, f_p$ above and set 
     $$B= (\cup_{i=1}^p A_i )\setminus \{1,\dots, p\}.$$
  For every $u\in B$ we pick a new variable $x_u$  of degree $\{u\}$ and set 
$$f_i'=f_i (\prod_{u\in B\cap A_i} x_u).$$
Now by the iterated application of  \ref{nonrad1} we know that the polynomials $f_1',\dots, f_p'$ generate a non-radical ideal and have degrees $A_1,\dots, A_p$ respectively.  A contradiction.

It follows from \ref{edgelabel} that (3) implies  (4). 

Equivalence of (5) and (6) follows from \ref{CS-K}, since the $K$-polynomial of a regular sequence generated by elements of degree $\A$ is 
\begin{equation} \label{polyF} 
\KK_\A(z)=\prod_{v=1}^s (1-\prod_{j\in A_v} z_j) \in \ZZ[z_1,\dots,z_n],
\end{equation} 
hence it depends only on $\A$. 

Now we prove that (4) implies (5). For $j\in[n]$, let $m_j=|\{v\in[s] : j\in A_v\}|$. Let $K$ be any field. By \ref{suppreg}  there exists a regular sequence $f_1,\dots, f_s$  in $S(m)$ of degrees $\A$. 
Set $I=(f_1,\dots,f_s)$. Then the K-polynomial of $S/I$ is as in \ref{polyF}. In order to prove  $I$ is a Cartwright-Sturmfels ideal, by \ref{JandE}, it is suffices to exhibit a monomial ideal $E$ in the polynomial ring $T=K[y_{1},y_{2},\dots, y_{n}]$ equipped with the (fine) $\ZZ^n$-grading $\deg y_i=e_i \in \ZZ^n$ such that: 
\begin{itemize} 
\item[(i)]  the $K$-polynomial  of $E$   is  $\KK^*_\A(z)$
\item[(ii)]  for every $j$, the largest exponent of $y_j$ in the generators of $E$ is bounded from above by $m_j$.
\end{itemize} 
  
We have  
$$\KK^*_\A(z)=\KK_\A(1-z_1,\dots, 1-z_n)= \prod_{v=1}^s G_v(z_1,\dots, z_n)$$ 
with 
$$G_v(z_1,\dots, z_n)=(1-\prod_{j\in A_v} (1-z_j) )=\sum_{\emptyset\neq B\subseteq A_v } (-1)^{|B|+1} \prod_{j\in B} z_j.$$
Notice that the ideal $E_v=( y_j : j\in A_v)$ is resolved by the truncated Koszul complex $\K^*(E_v)$ associated  to the variables $y_j$ with  $j\in A_v$ and hence its $\ZZ^n$-graded Hilbert series is $$\frac{G_v(z_1,\dots, z_n)}{\prod_{i=1}^n (1-z_i)}.$$ 
Indeed the full Koszul complex $\K(E_v)$ resolves $T/E_v$ and we take the truncated Koszul complex obtained from $\K(E_v)$  by removing the component in homological position $0$ and shifting the remaining homological positions by $1$  to get a $T$-resolution of the ideal $E_v$. 
 
We claim that, under the assumption on $G(\A)$ from (4),  the ideal $$E=\prod_{v=1}^s  E_v$$ has the properties (i) and (ii). Indeed, (ii) is satisfied by construction. We now prove (i).  
 
The minimal free resolution of the product of any collection of ideals $I_1,\dots, I_v$  generated by linear forms is described in \cite{CT}. It is proved that  such a resolution is obtained as a subcomplex,  supported on a specific polymatroid, of the tensor product of the truncated Koszul complexes $\K^*(I_j)$.   We notice that assumption (4) on $G(\A)$ is equivalent to the assumption that the ideal $E$ has $\prod_{v=1}^s  |A_v|$ generators. 
 Then the results in \cite{CT} show that the minimal free resolution of $E$ is the tensor product  $\K^*(E_1)\otimes \K^*(E_2) \otimes \cdots \otimes \K^*(E_v)$. This implies (i) and concludes the proof that (4) implies (5). 
 
Since  every Cartwright-Sturmfels ideal is radical, then (7) implies (1).
 
 It remains to prove that (5) implies (7).  Let $K$ be any field, $S=K[x_{ij}\mid 1\leq j\leq n,\ 1\leq i\leq m_j]$ be a $\ZZ^n$-graded polynomial ring over $K$ and $I$ an ideal generated by elements $f_1,\dots, f_s$ of degrees $A_1,\dots, A_s$.   For $i=1,\dots, s$ we introduce new variables $t_{ij}$  for every pair $i\in [s]$ and $j\in [n]$ such that $j\in A_i$, with $\deg t_{ij}=e_j\in \ZZ^n$. Then in the $\ZZ^n$-graded polynomial ring  
$$R=S[t_{ij} :  i\in [s] \ \  j\in A_i]$$ we consider the polynomials $$g_i=f_i+\prod_{j\in A_i} t_{ij}$$ of degree $A_i$ and observe that the $g_1,\dots, g_s$ form  a regular sequence. 
Indeed, the leading term of $g_i$ with respect to the lexicographic order with the $t$'s larger than the $x$'s is $\prod_{j\in A_i} t_{ij}$. Since the leading terms of the $g_1,\dots, g_s$ are pairwise coprime, we have that they form a Gr\"obner basis and $g_1,\dots, g_s$   form a regular sequence of elements of degree $\A$ is the  $\ZZ^n$-graded polynomial ring $R$. Hence $(g_1,\dots, g_s)\in \CS(R)$ by assumption.   The $g_i$'s specialize to  to the $f_i$'s modulo the  multigraded ideal of linear forms $(t_{ij} )$.  Hence $I\in \CS(S)$  by \ref{section}. 
  \end{proof}

\end{document}